\documentclass[reqno, 11pt, fleqn]{amsart}
\usepackage{graphicx, enumitem}
\usepackage{amsmath,amssymb,amsthm, mathrsfs, mathtools}
\usepackage{diagbox}
\usepackage[usenames,dvipsnames]{xcolor}
\usepackage[backref]{hyperref}
\usepackage{xcolor}
\hypersetup{
	colorlinks,
	linkcolor={red!60!black},
	citecolor={green!60!black},
	urlcolor={blue!60!black}
}
\usepackage[T1]{fontenc}
\usepackage{lmodern}
\usepackage[babel]{microtype}
\usepackage[english]{babel}
\usepackage{comment}

\newcommand{\cH}{\mathcal{H}}

\usepackage{geometry}
\numberwithin{equation}{section}
\numberwithin{figure}{section}

\usepackage{enumitem}

\theoremstyle{plain}
\newtheorem{thm}{Theorem}[section]
\newtheorem{theorem}[thm]{Theorem}
\newtheorem{prop}[thm]{Proposition}

\newtheorem{clm}[thm]{Claim}

\newtheorem{lemma}[thm]{Lemma}
\newtheorem{question}[thm]{Question}

\theoremstyle{definition}
\newtheorem{definition}[thm]{Definition}

\theoremstyle{remark}

\DeclarePairedDelimiter{\parens}{(}{)}
\DeclarePairedDelimiter{\set}{\{}{\}}

\title{Ramsey equivalence for asymmetric pairs of graphs}

\author{Simona Boyadzhiyska
	\and
	Dennis Clemens
	\and
	Pranshu Gupta
	\and
	Jonathan Rollin
}

\address{Institut f\"ur Mathematik, Freie Universit\"at Berlin, Berlin, Germany}
\email{s.boyadzhiyska@fu-berlin.de}
\thanks{The first author was supported by the Deutsche Forschungsgemeinschaft (DFG) Graduiertenkolleg ``Facets of Complexity'' (GRK 2434).}
\address{Hamburg University of Technology, Institute of Mathematics, Hamburg, Germany}
\email{dennis.clemens@tuhh.de, pranshu.gupta@tuhh.de}
\address{FernUniversität in Hagen, Fakult\"at f\"ur Mathematik und Informatik, Hagen, Germany}
\email{jonathan.rollin@fernuni-hagen.de}

\begin{document}
	\begin{abstract}
		A graph $F$ is Ramsey for a pair of graphs $(G,H)$ if any red/blue-coloring of the edges of $F$ yields a copy of $G$ with all edges colored red or a copy of $H$ with all edges colored blue.
		Two pairs of graphs are called Ramsey equivalent if they have the same collection of Ramsey graphs.
		The symmetric setting, that is, the case $G=H$, received considerable attention. This led to the open question whether there are connected graphs $G$ and $G'$ such that $(G,G)$ and $(G',G')$ are Ramsey equivalent.
		We make progress on the asymmetric version of this question and identify several non-trivial families of Ramsey equivalent pairs of connected graphs.
		
		Certain pairs of stars provide a first, albeit trivial, example of Ramsey equivalent pairs of connected graphs.
		Our first result characterizes all Ramsey equivalent pairs of stars. The rest of the paper focuses on pairs of the form $(T,K_t)$, where $T$ is a tree and $K_t$ is a complete graph. We show that, if $T$ belongs to a certain family of trees, including all non-trivial stars, then $(T,K_t)$ is Ramsey equivalent to a family of pairs of the form $(T,H)$, where $H$ is obtained from $K_t$ by attaching disjoint smaller cliques to some of its vertices. In addition, we establish that 
		for $(T,H)$ to be Ramsey equivalent to $(T,K_t)$,
		$H$ must have roughly this form. On the other hand, we prove that for many other trees $T$, including all odd-diameter trees, $(T,K_t)$ is not equivalent to any such pair, not even to the pair $(T, K_t\cdot K_2)$, where $K_t\cdot K_2$ is a complete graph $K_t$ with a single edge attached.
	\end{abstract}
	
	\maketitle 
	
	\section{Introduction}\label{sec:intro}
	
	\subsection{Ramsey equivalence}
	We say that a graph $F$ is \emph{Ramsey} for another graph $H$, if any red/blue-coloring of the edges of $F$ yields a copy of $H$ all of whose edges have the same color, that is, a \emph{monochromatic} copy of $H$; we write $\mathcal{R}(H)$ for the set of all Ramsey graphs for $H$. The seminal result of Ramsey~\cite{Ram30} establishes that $\mathcal{R}(H)\neq \emptyset$ for any graph $H$.
	Characterizing the graphs in $\mathcal{R}(H)$ exactly is a hard task accomplished for only few small graphs so far (see for example~\cite{BSS05,burr1976graphs}).
	A natural next problem then is
	to investigate the properties of the graphs belonging to $\mathcal{R}(H)$ for a given $H$.

	The most prominently studied property is the smallest number of vertices among all graphs in $\mathcal{R}(H)$ for a given $H$. This quantity is called the \emph{Ramsey number} of $H$. Ramsey numbers have proven to be notoriously hard to compute in many cases and have attracted a lot of attention over the years;
	see e.g. the survey~\cite{conlon2015recent} and the references therein.
	In the 1970s researchers initiated the study of other graph parameters in the context of Ramsey graphs, like the number of edges, the clique number, and the minimum degree (see for example~\cite{burr1976graphs,erdHos1978size,folkman1970graphs}). 
	A natural question to ask when studying problems of this type is: how do these values change when we modify the target graph $H$ slightly? More generally, how does the collection of Ramsey graphs change? This question motivated Szab\'o, Zumstein, and Z\"urcher~\cite{szabo2010minimum} to define the notion of Ramsey equivalence.
	
	\begin{definition}\rm
		Two graphs $H$ and $H'$ are \emph{Ramsey equivalent}, denoted $H\sim H'$, if $\mathcal{R}(H) = \mathcal{R}(H')$.
	\end{definition}
	
	It is not difficult to show that Ramsey equivalent pairs of graphs exist: for instance, the graph obtained by adding an isolated vertex to the clique $K_t$ for $t\geq 3$ is Ramsey equivalent to $K_t$. Szab\'o, Zumstein, and Z\"urcher~\cite{szabo2010minimum} found further examples of disconnected graphs that are Ramsey equivalent to the clique $K_t$ (see also~\cite{bloom2018ramsey,fox2014ramsey} for further results in this direction) and asked whether there exist any \emph{connected} such graphs.
	Perhaps surprisingly, several years later Fox, Grinshpun, Liebenau, Person, and Szab\'o~\cite{fox2014ramsey} settled this question
	in the negative: they showed that no connected graph is Ramsey equivalent to $K_t$.
	In light of this result, they raised the following question:
	
	\begin{question}[\cite{fox2014ramsey}]\label{quest:main}
		Is there a pair of non-isomorphic connected graphs that are Ramsey equivalent?
	\end{question}
	
	The above question is wide open.
	Not much is known even in the special case where the two graphs differ only by a pendent edge.
	It was shown by Clemens, Liebenau, and Reding~\cite{clemens2020minimal} that no pair of 3-connected graphs can be Ramsey equivalent. A result from Grinsphpun's PhD thesis~\cite[Lemma 2.6.3.]{grinshpun2015some} allows us to show non-equivalence for further pairs consisting of a graph $H$ and the graph $H$ with a pendent edge. Further evidence that the answer to Question~\ref{quest:main} might be negative is provided for example in~\cite{axenovich2017conditions,savery2022chromatic}.
	\medskip

	In this paper, we study Ramsey equivalence in the asymmetric setting and explore a variant of Question~\ref{quest:main}. We begin by defining the necessary notions. We say that a graph $F$ is \emph{Ramsey} for a pair of graphs $(G,H)$, and write $F\rightarrow (G,H)$, if, for every red/blue-coloring of $E(F)$ there exists a red copy of $G$, that is, a copy of $G$ with all edges colored red, or a blue copy of $H$, defined similarly.
	We denote the collection of all Ramsey graphs for $(G,H)$ by $\mathcal{R}(G,H)$. We call a graph $F\in \mathcal{R}(G,H)$ \emph{Ramsey-minimal} for $(G, H)$ if no proper subgraph of $F$ is Ramsey for $(G, H)$, and we denote the corresponding collection by $\mathcal{M}(G,H)$.
	
	\begin{definition}\rm
		We call two pairs of graphs $(G, H)$ and $(G', H')$ \emph{Ramsey equivalent}, denoted $(G,H)\sim (G', H')$, if $\mathcal{R}(G,H) = \mathcal{R}(G', H')$.
	\end{definition}
	
	Note that we can define Ramsey equivalence also in terms of the corresponding Ramsey-minimal graphs, requiring that $\mathcal{M}(G,H) = \mathcal{M}(G',H')$. Our goal is to explore the notion of Ramsey equivalence for asymmetric pairs of connected graphs and in particular the asymmetric version of Question~\ref{quest:main}.
	Previously known results allow us to exclude some potential candidates.
	Let $\omega(G)$ denote the clique number of a graph $G$, defined as
	the largest integer $n$ such that $K_n$ is a subgraph of~$G$.
	A famous result of Nešetřil and Rödl~\cite{NESETRIL1976243} establishes that, for every graph $G$, there is a Ramsey graph for $G$ that has the same clique number as $G$. Hence, the disjoint union of $G$ and $H$ has a Ramsey graph $F$ with clique number $\max\{\omega(G),\omega(H)\}$ and this graph $F$ is also a Ramsey graph for~$(G,H)$. This gives the following statement which we shall use several times in our proofs.
	
	\begin{theorem}[\cite{NESETRIL1976243}]\label{thm:NRcliquenumber}
		Each pair $(G,H)$ of graphs has a Ramsey graph with clique number equal to $\max\{\omega(G),\omega(H)\}$.
	\end{theorem}
	
	This result implies that, if $(G,H)\sim (G',H')$, then $\max\{\omega(G),\omega(H)\} = \max\{\omega(G'),\omega(H')\}$. As a second example, Savery~\cite[Section 3.1]{savery2022chromatic} proved that $(G,H)\not\sim(G',H')$ for all graphs $G$, $H$, $G'$, and $H'$ with $\chi(G)+\chi(H)\neq \chi(G')+\chi(H')$, where $\chi(G)$ denotes the chromatic number of $G$.
	
	It turns out, however, that the asymmetric version of Question~\ref{quest:main} has an affirmative answer. Let $K_{1,s}$ denote a star with $s$ edges.
	Through a simple application of Petersen's Theorem~\cite{petersen}, Burr, Erd\H{o}s, Faudree, Rousseau, and Schelp~\cite{burr1981starforests} showed that, for any odd integers $r$, $s\geq 1$, the only Ramsey-minimal graph for the pair of stars $(K_{1,r}, K_{1,s})$ is the star $K_{1,r+s-1}$. Thus, any odd integers $r,s,r',s'\geq 1$ with $r+s = r'+s'$ satisfy  $(K_{1,r}, K_{1,s}) \sim (K_{1,r'}, K_{1,s'})$. This example is perhaps not very satisfying, as pairs of odd stars have only a single Ramsey-minimal graph. It is then interesting to ask whether there are any Ramsey equivalent pairs of connected graphs with a larger, maybe even an infinite number of Ramsey-minimal graphs. Our main result shows that the answer is yes, exhibiting an infinite family of Ramsey equivalent pairs of connected graphs of the form $(T,K_t) \sim (T,K_t\cdot K_2)$, where $T$ is a certain kind of tree.
	
	\subsection{Results}

	In light of the discussion in the previous paragraph, one might ask whether there exist any other pairs of stars that are Ramsey equivalent. In our first result, we answer this question negatively. Note that $\mathcal{M}(K_{1,r}, K_{1,s})$ is infinite whenever $rs$ is even~\cite{burr1981starforests}.
	
	\begin{theorem}\label{thm:Stars}
		Let $a$, $b$, $x$, $y$ be positive integers with $\{a,b\} \neq \{x,y\}$. Then $(K_{1,a},K_{1,b})\sim(K_{1,x},K_{1,y})$ if and only if $a+b=x+y$ and $a$, $b$, $x$, and $y$ are odd.
	\end{theorem}
	
	Note that each pair of stars has a Ramsey graph that is a star.
	Since stars cannot be Ramsey graphs for other connected graphs than (smaller) stars, pairs of stars are not Ramsey equivalent to pairs of connected graphs that are not star pairs.
	
	\smallskip
	
	We next study Ramsey equivalence for pairs of the form $(T, K_t)$, where $T$ is a tree and $t\geq 3$.
	Note that in the case where $T$ is a single vertex or edge the collection of Ramsey graphs is trivial, as $\mathcal{M}(K_1,K_t)=\{K_1\}$ and $\mathcal{M}(K_2,K_t)=\{K_t\}$.
	From now on, unless otherwise specified, we will assume that $T$ has at least two edges. It was shown by {\L}uczak~\cite{luczak1994ramsey} that in this case $\mathcal{M}(T,K_t)$ is infinite.
	Perhaps surprisingly, we find non-trivial Ramsey equivalent pairs in this setting. To describe some of those pairs, we need the following definitions.
	For integers $a\geq 1$, $b\geq 2$, and $t\geq 3$ with $a\leq t$, let $K_t\cdot aK_b$ denote the graph consisting of a copy of $K_t$ and $a$ pairwise vertex-disjoint copies of $K_b$, each sharing exactly one vertex with the copy of $K_t$ (see Figure~\ref{fig:3suitableCaterpillar} left for an example).
	We call a tree $T$ an \emph{($s$-)suitable caterpillar}, if $T$ consists of a path $P$ on three vertices and up to $3s-1$ further vertices of degree 1  such that the endpoints of $P$ are of degree exactly $s+1$ in $T$ and the middle vertex of $P$ is of degree at most $s+1$ in $T$ (see Figure~\ref{fig:3suitableCaterpillar} right).
	
	\begin{figure}
		\centering
		\includegraphics{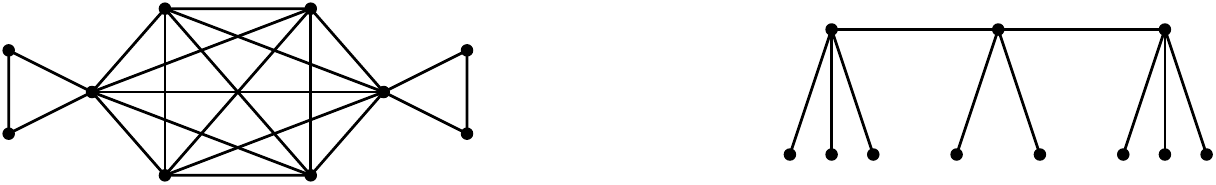}
		\caption{The graph $K_6\cdot 2K_3$ (left) and the largest $3$-suitable caterpillar (right).}
		\label{fig:3suitableCaterpillar}
	\end{figure}
	
	\begin{theorem}\label{thm:equivalentpairs}\hfill
		\begin{enumerate}[label = (\alph*)]
			\item For all integers $s\geq 2$ and $t\geq 3$, we have $(K_{1,s},K_t)\sim(K_{1,s},K_t\cdot K_2)$.\label{thm:StarClique}
			\item Let $a\geq 1$ and $b\geq 2$ be integers, and let $T$ be a star with at least two edges or a suitable caterpillar.
			For any sufficiently large $t$, we have $(T,K_t)\sim (T,K_t\cdot aK_b)$.\label{thm:StarCaterpillar}
		\end{enumerate}
	\end{theorem}
	
	Observe that the first part of the above theorem holds for each $t\geq 3$, while we need a sufficiently large $t$ to prove the second part.
	Our proof shows that we can take $t$ to be quadratic in the parameters $a$, $b$ and $\lvert V(T)\rvert$ (see Proposition~\ref{prop:woven}, where we use the fact that the Ramsey number of $(T,K_t)$ is $(\lvert V(T)\rvert-1)(t-1)+1$, as shown by Chvátal~\cite{TreeCompleteRamseyNumber}). Since we do not have a non-trivial lower bound, we make no effort to optimize this value, choosing to present a simpler proof instead.
	\smallskip
	
	We complement the equivalence result above by proving Ramsey non-equivalence for several other families of pairs of trees and cliques. Theorem~\ref{thm:NRcliquenumber} shows that we may restrict our attention to pairs $(G,H)$ with $\max\set{\omega(G),\omega(H)} = t$, since otherwise $(G,H)\not\sim (T,K_t)$.
	Before we state the result, we again need some definitions.
	The length of a path is its number of edges.
	The diameter $\mathrm{diam}(T)$ of a tree $T$ is the length of its longest path.
	Trees of even diameter contain a unique central vertex, that is, a vertex that is the middle vertex in each longest path.
	Let $\mathcal{T}$ denote the class of all trees $T$ of diameter at least three such that:
	\begin{itemize}
		\item if $\mathrm{diam}(T)$ is even, the neighbors of the central vertex of $T$ are of degree at most two, and
		\item if $\mathrm{diam}(T)=4$, the central vertex is  of degree at least $3$.
	\end{itemize}
	Note in particular that this class contains all trees of odd diameter.
	
	Theorem~\ref{thm:equivalentpairs} above shows that for some trees $T$ certain modifications to the pair $(T,K_t)$ yield a Ramsey equivalent pair.
	Specifically, the Ramsey graphs do not change when we attach certain disjoint pendent graphs to the second component of the pair, that is, at the vertices of $K_t$.
	The first part of the following theorem states that this behavior does not generalize to trees from the family $\mathcal{T}$ defined above in a strong sense: for each tree $T\in\mathcal{T}$ and each $t\geq 3$, we have $(T,K_t)\not\sim (T,K_t\cdot K_2)$.
	The second part shows that in order to obtain a Ramsey equivalent pair (like in Theorem~\ref{thm:equivalentpairs}\ref{thm:StarCaterpillar}) it is necessary that the graphs attached to different vertices of $K_t$ do not intersect.
	Finally, we consider modifications to the first component of the pair, namely $T$.
	If $T$ and $T'$ are trees of different sizes, we have $(T,K_t)\not\sim (T',K_t)$, since the Ramsey numbers of these pairs differ, as shown by Chvátal~\cite{TreeCompleteRamseyNumber}.
	The third part of the theorem below shows that $T$ cannot be replaced by \emph{any} other connected graph~$G$ if the second component of the pair stays unchanged. Note that the Ramsey number alone is not sufficient to distinguish certain pairs $(T, K_t)$ and $(G, K_t)$: for example, Keevash, Long, and Skokan~\cite{keevash2021cycle} showed that when $\ell =\Omega\parens*{\frac{\log t}{\log\log t}}$ the Ramsey numbers of $(C_\ell, K_t)$ and $(T_\ell, K_t)$ are the same, where $C_\ell$ and $T_\ell$ denote a cycle and a tree on $\ell$ vertices, respectively.
	
	\begin{theorem}\label{thm:non-equivalentpairs}
		Let $T$ be a tree, $t\geq 3$ be an integer, and let $G$ and $H$ be graphs with $(G,H)\neq(T,K_t)$. Then $(T,K_t)\not\sim (G,H)$ if one of the following conditions holds:
		\begin{enumerate}[label = (\alph*)]
			\item  $G=T$, $T\in\mathcal{T}$, and $H$ is connected,\label{thm:diameterTrees}
			
			\item $H$ contains a copy $K$ of $K_t$, and $H$ contains a cycle with vertices from both $V(K)$ and $V(H)\setminus V(K)$,\label{thm:noexternalcycle}
			
			\item $G\neq T$, $G$ is connected, and $H=K_t$.\label{thm:connectedG}
		\end{enumerate}
	\end{theorem}
	
	As we will see in Section~\ref{sec:nonequivalent}, our construction actually allows us to prove the statement from the first part of the theorem above for a larger class of trees. Again, since our results do not lead to a complete characterization of those trees $T$ for which $(T, K_t)\sim (T, K_t\cdot K_2)$, we choose to state the simpler, albeit somewhat weaker, result here. 
	As a specific example, note that our results imply that for sufficiently large $t$ and a path $P$  we have $(P,K_t)\sim (P,K_t\cdot K_2)$ if and only if $P$ has two or four edges.
	
	In this paper, we study what pairs of connected graphs $(G,H)$ can be Ramsey equivalent to pairs of the form $(T,K_t)$.
	A summary of our results is given in Table~\ref{tab:summary}.	
	We focus on the two cases $G=T$ and $H=K_t$. It would be interesting to know whether there are any pairs $(G,H)$ of connected graphs with $G\subsetneq T$ and $K_t\subsetneq H$ that are Ramsey equivalent to $(T, K_t)$.

	\begin{table}
		\begin{tabular}{r|l|l|l|l}			
			\diagbox[width=3em]{$G$}{$H$} & \multicolumn{1}{c|}{$\omega(H)\neq t$} & \multicolumn{1}{c|}{$H=K_t$} & \multicolumn{1}{c|}{\renewcommand{\arraystretch}{0.75}\begin{tabular}[c]{@{}c@{}}\tiny $H\supsetneq K_T$ and\\ \tiny$H\text{-}E(K_t)$ has\\\tiny $t$ components\end{tabular}} & \multicolumn{1}{c}{\renewcommand{\arraystretch}{0.75}\begin{tabular}[c]{@{}c@{}}\tiny $H\supsetneq K_T$ and\\\tiny $H\text{-}E(K_t)$ has\\\tiny$<t$ components\end{tabular}}\\
			\hline
			$G=T$ & $\not\sim$ (\cite{NESETRIL1976243}) & $\sim$ (triv.) & \renewcommand{\arraystretch}{0.75}\begin{tabular}[c]{@{}r@{~}l@{}} $\sim$ & for some $H$ if $T$ star or suit.\ cat.\ (\ref{thm:equivalentpairs})\\
				$\not\sim$ &if $T\in\mathcal{T}$ (\ref{thm:non-equivalentpairs}\ref{thm:diameterTrees}) \\ \multicolumn{2}{l}{otherwise partial results}\end{tabular} & $\not\sim$ (\ref{thm:non-equivalentpairs}\ref{thm:noexternalcycle}) \\
			\hline
			$G\neq T$ & $\not\sim$ (\cite{NESETRIL1976243}) & $\not\sim$ (\ref{thm:non-equivalentpairs}\ref{thm:connectedG}) & \multicolumn{1}{c|}{open} & $\not\sim$ (\ref{thm:non-equivalentpairs}\ref{thm:noexternalcycle})
		\end{tabular}
		\vspace{4mm}
		\caption{
			Known results about Ramsey equivalence between $(T,K_t)$ and $(G,H)$ where $t\geq 3$, $T$ is a tree, $G$ and $H$ are connected graphs, and $\omega(G)\leq \omega (H)$. Each entry states whether the respective pairs $(T,K_t)$ and $(G,H)$ are Ramsey equivalent or not.}		
		\label{tab:summary}
	\end{table}

	\subsection{Notation}
	
	Given a graph $G$, we denote its vertex set and its edge set by $V(G)$ and $E(G)$, respectively.
	For a set $X\subseteq V(G)$  we write $G-X$ for the graph obtained from $G$ by removing the vertices in $X$ and all their incident edges; for a single vertex $x\in V(G)$, we write $G-x=G-\{x\}$;
	similarly for a subgraph $F$ of $G$ we let $G-F=G-V(F)$.
	For a set $Y\subseteq E(G)$, we write $G-Y$ for the graph obtained from $G$ by removing the edges in $Y$; for a single edge $e\in E(G)$, we write $G-e=G-\{e\}$.
	Throughout the paper unless otherwise specified a coloring is meant to be an edge-coloring of the given graph $G$. As we always call the two colors red and blue,
	we use red/blue-coloring and 2-coloring as synonyms
	of each other. Given any two graphs $H_1$ and $H_2$,
	we say that a 2-coloring is $(H_1,H_2)$-free, if there is no red copy of $H_1$ and no blue copy of $H_2$.

	\subsection{Organization of the paper}
	In Section~\ref{sec:stars}, we prove Theorem~\ref{thm:Stars}.
	Section~\ref{sec:equivalent} contains the proof of our main equivalence result, namely Theorem~\ref{thm:equivalentpairs}, and in Section~\ref{sec:nonequivalent} we prove Theorem~\ref{thm:non-equivalentpairs} on Ramsey non-equivalent pairs.

	
	\section{Star pairs}\label{sec:stars}
	In this section, we prove Theorem~\ref{thm:Stars}.
	We note that his theorem can be deduced from Theorem~1 in~\cite{NiRa98}. However, the calculations are tedious and 
	for completeness we present explicit constructions here
	when we want to show that there exist
	graphs that are Ramsey for certain pairs of stars and not Ramsey for other star pairs.
	
	\smallskip
	
	Observe that, given positive integers $a$ and $b$, an $(a+b-2)$-regular graph $F$ is a Ramsey graph for a pair $(K_{1,a},K_{1,b})$ if and only if $E(F)$ cannot be decomposed into an $(a-1)$-regular subgraph and a $(b-1)$-regular subgraph.
	A $k$-regular spanning subgraph is also called a \emph{$k$-factor}.
	Our results rely on the rich theory on factors.
	Specifically we need the following fact.

	\begin{lemma}\label{lem:regularFactors}
		Let $p$, $q$ and $r$ be integers, with $p$ and $q$ being odd.
		Further, assume that $p<q\leq r$ if $r$ is odd, and that $p<q\leq r/2$
		if $r$ is even.
		There is an $r$-regular graph that has a $q$-factor 
		and no $p$-factor.
	\end{lemma}
	
	In order to prove the above lemma, we apply a theorem due to Belck~\cite{Belck} (a special case of the well-known $f$-factor theorem of Tutte~\cite{Tutte52}), which provides a necessary and sufficient condition for the existence of $k$-factors in regular graphs.
	For a graph $G$ and a set $D\subseteq V(G)$, we call a component $C$ of $G-D$ an \emph{odd component} with respect to~$D$
	if $|V(C)|$ is odd, and we let $q_G(D)$ denote the number of such components.
	We use the following corollary of Theorem \textrm{IV} from~\cite{Belck}.
	
	\begin{thm}[\cite{Belck}]\label{thm:Tutte-pair}
		Let $G$ be a graph and let $p>0$ be an odd integer. If there exists a set $D\subseteq V(G)$ such that 
		$p|D| < q_G(D)$, then $G$ has no $p$-factor.  
	\end{thm}

	\begin{proof}[Proof of Lemma~\ref{lem:regularFactors}]
		Given $p$, $q$ and $r$ as described in 
		the statement, we aim to construct an
		$r$-regular graph $F$ that has a $q$-factor 
		and no $p$-factor. 
		The graph $F$ will be constructed in three steps.
		
		In the first step, we find an $r$-regular graph
		$G$ with an even number of vertices 
		that has a $q$-factor $G_{q}$ and a matching $M_G$
		with $\lfloor (r-1)/2 \rfloor$ edges
		which contains exactly $(q-1)/2$ edges of $G_q$. 
		To this end, define $G$ to be the graph obtained by taking $2r(r-q+1)$ copies $Q_{i,j}$ of $K_{q+1}$, with 
		$i\in[r-q+1]$ and $j\in [2r]$, 
		and adding a perfect matching between any two copies 
		$Q_{i_1,j}$ and $Q_{i_2,j}$ with $i_1\neq i_2$ and 
		$j\in [2r]$. Then $G$ has an even number of vertices and is $r$-regular.
		Moreover, the subgraph $G_q$ that consists of all $Q_{i,j}$ 
		is a $q$-factor. The matching $M_G$ can be found
		by taking $\frac{q-1}{2}$ independent edges from $Q_{1,1}$ and
		one edge from every matching between 
		$Q_{1,j}$ and $Q_{2,j}$ with 
		$2\leq j\leq \lfloor \frac{r-q+2}{2} \rfloor$.
		Set $M_q:=M_G\cap E(G_q)$. 
		
		\smallskip
		For the second step, let $H$ and $H_q$ denote 
		the graphs obtained from $G$, respectively $G_q$, by adding a new vertex $u$ and replacing every edge $vw\in M_G$, respectively $vw\in M_q$, by the edges $uv$ and $uw$. Then $H$ has an odd number of vertices, 
		$u$ is of degree $2\lfloor (r-1)/2 \rfloor$ in $H$, and all other vertices are of degree $r$. Moreover, $H_q$ is a spanning subgraph of $H$ in which $u$ is of degree $q-1$ and all other vertices are of degree $q$.

		\smallskip	
		For the third step, we consider two cases depending on the parity of $r$.
		
		{\bf Case 1: $\mathbf{r}$ is odd.} In this case $u$ has degree $r-1$ in the graph $H$. Let $t=r-q+1$, and let $F=F(q,r)$ denote the graph obtained from a copy of $K_{t}$ with vertex set $D=\{d_j: j \in [t]\}$ and $qt$ vertex disjoint copies $H^1,\ldots,H^{qt}$ of the graph $H$ as follows:
		For each $i\in [qt]$, let $u^i$ denote the copy of $u$ in $H^i$. 
		We partition the set $\{u^i:i\in [qt]\}$ into $t$ sets $U_1,U_2,\ldots,U_t$ each of size $q$
		and, for each $j\in [t]$, add an edge between $d_j$ and each vertex in $U_j$.
		An illustration of the construction is given in Figure~\ref{fig:regularFactor} (left).
		\begin{figure}
			\centering
			\includegraphics{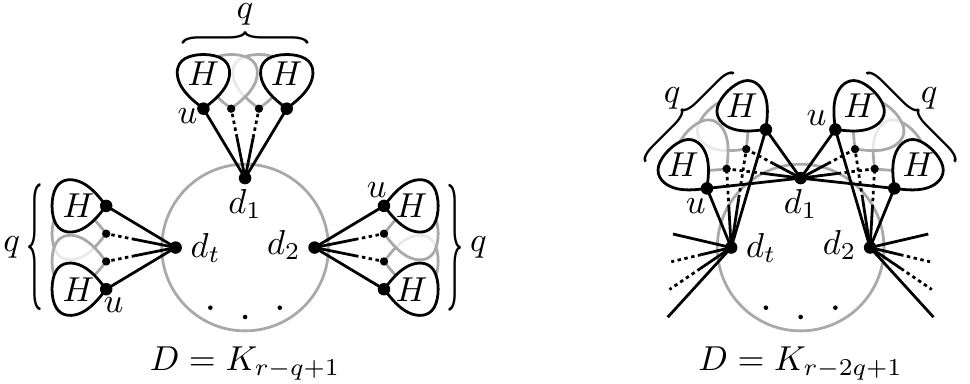}
			\caption{A construction of an $r$-regular graph with a $q$-factor and no $p$-factor for odd $r$ (left) and even $r$ (right).}
			\label{fig:regularFactor}
		\end{figure}		
		Then $F$ is $r$-regular.
		Moreover, $F$ has a $q$-factor, given by the subgraph consisting of all copies of $H_q$ (coming from the $H^i$ with $i\in [qt]$) and all edges between $D$ and the copies of $u$.
		
		It thus remains to show that $F$ does not admit a $p$-factor.
		This follows from Theorem~\ref{thm:Tutte-pair}. Indeed,  the odd components of $F-D$ 
		are exactly the $q(r-q+1)$ copies of $H$, and therefore
		$$p|D| - q_G(D)=p\lvert D \rvert - q(r-q+1) = (p-q)(r-q+1) < 0\, ,$$ 
		since $p < q \leq r$ by assumption.
		
		\smallskip
		
		{\bf Case 2: $\mathbf{r}$  is even.} In this case $u$ has degree $r-2$ in the graph $H$. Moreover, by assumption we have $q\leq r/2$.
		Let $t=r-2q+1$, and let $F=F(q,r)$ denote the graph obtained from a copy of $K_{t}$ with vertex set $D=\{d_1,d_2,\ldots,d_t\}$  and	 $qt$ vertex disjoint copies $H^1,\ldots,H^{qt}$ of the graph $H$ as follows:
		For each $i\in [qt]$, let $u^i$ denote the copy of $u$ in $H^i$. 
		We partition the set $\{u^i:i\in [qt]\}$
		into $t$ sets $U_1,U_2,\ldots,U_t$ each of size $q$
		and, for each $j\in [t]$, add an edge between $d_j$ and each vertex in $U_j\cup U_{j+1}$, where $U_{t+1}:=U_1$. This is illustrated in Figure~\ref{fig:regularFactor} (right).
		
		Then the graph $F$ is $r$-regular. Moreover, it contains a $q$-factor consisting of all copies of $H_q$ (coming from the $H^i$ 
		with $i\in [qt]$) and all edges
		between $d_j$ and $U_j$ for every $j\in [t]$. Furthermore, by Theorem~\ref{thm:Tutte-pair},
		$F$ does not have a $p$-factor. Indeed, the odd components of $F-D$ 
		are exactly the $qt$ copies of $H$, and therefore
		$$p|D| - q_G(D)=p\lvert D \rvert - qt = 
		(p-q)t < 0\, ,$$ 
		since $p< q$ and $q\leq r/2$ and hence $t\geq 1$ by assumption.
	\end{proof}
	
	\begin{proof}[Proof of Theorem~\ref{thm:Stars}]
		First observe that $K_{1,a+b-1}$ is a Ramsey graph for $(K_{1,x},K_{1,y})$ if and only if $x+y\leq a+b$.
		This shows that $(K_{1,a},K_{1,b})\not\sim(K_{1,x},K_{1,y})$ when $a+b\neq x+y$.
		For the remainder of the proof assume that $a+b=x+y$.

		As discussed in the introduction, if $a$ and $b$ are both odd, then $K_{1,a+b-1}$ is the unique minimal Ramsey graph for $(K_{1,a}, K_{1,b})$~\cite{burr1981starforests}.  So $(K_{1,a},K_{1,b})\sim(K_{1,x},K_{1,y})$ if $a$, $b$, $x$, $y$ are all odd.
		It remains to consider the case where at least one of $a$, $b$, $x$, and $y$ is even and find a distinguishing graph, that is, a graph that is Ramsey for one of the pairs of stars and not Ramsey for the other pair.
		Without loss of generality, assume that $a$ is the largest even number in $\{a,b,x,y\}$.
		Let $r=a+b-2=x+y-2$.
		Recall that an $r$-regular graph is a Ramsey for $(K_{1,a},K_{1,b})$ if and only if it has no $(a-1)$-factor. We consider several cases.

		\textbf{Case 1: $\mathbf{xy}$ is odd.} Then each Ramsey graph for $(K_{1,x},K_{1,y})$ contains $K_{1,a+b-1}$, as remarked above, and hence no graph of maximum degree at most $r$ is a Ramsey graph for $(K_{1,x},K_{1,y})$.
		Consider an $r$-regular graph on an odd number of vertices, which exists since $r$ is even.
		Since $(a-1)$ is odd, this graph does not have an $(a-1)$-factor and is therefore  Ramsey for $(K_{1,a},K_{1,b})$ but not Ramsey for $(K_{1,x},K_{1,y})$.
		Thus $(K_{1,a},K_{1,b})\not\sim(K_{1,x},K_{1,y})$.
		
		\textbf{Case 2: $\mathbf{xy}$ is even.} We may assume that $x$ is the larger even number in 
		$\{x,y\}$. Then $a>x$, 
		since $a$ is the largest even number in 
		$\{a,b,x,y\}$, and since
		$\{a,b\} \neq \{x,y\}$. We again distinguish two  cases. 
		
		\textbf{Case 2.1: $\mathbf{b}$ is odd.} Then $r$ is odd. Setting $q=a-1$ and $p=x-1$,
		we know that $p$ and $q$ are odd, and $p<q\leq r$.
		Hence, using Lemma~\ref{lem:regularFactors}, we find
		an $r$-regular graph $F$ that has an $(a-1)$-factor and no $(x-1)$-factor.
		Thus  $F\not\to(K_{1,a},K_{1,b})$ and $F\to(K_{1,x},K_{1,y})$.
		Hence $(K_{1,a},K_{1,b})\not\sim(K_{1,x},K_{1,y})$.
		
		\textbf{Case 2.1: $\mathbf{b}$ is even.} Then $r$ and $y$ are even.
		As we have $a>x\geq y$ and $a+b-2=r=x+y-2$, 
		we obtain $b<y\leq \frac{r+2}{2}$. 
		Setting $q=y-1$ and $p=b-1$,
		we know that $p$ and $q$ are odd, and $p<q\leq \frac{r}{2}$.
		Hence, using Lemma~\ref{lem:regularFactors}, we find
		an $r$-regular graph $F$ that has a $(y-1)$-factor
		and no $(b-1)$-factor. Thus $F\not\to(K_{1,x},K_{1,y})$ and
		$F\to(K_{1,a},K_{1,b})$. Hence $(K_{1,a},K_{1,b})\not\sim(K_{1,x},K_{1,y})$.		
	\end{proof}

	\section{Equivalence results for trees and cliques}\label{sec:equivalent}
	
	\begin{proof}[Proof of Theorem~\ref{thm:equivalentpairs}\ref{thm:StarClique}] If $F\to (K_{1,s},K_t\cdot K_2)$, then also $F\to (K_{1,s},K_t)$. It suffices to show that, if $F\not\to (K_{1,s},K_t\cdot K_2)$, then also $F\not\to(K_{1,s},K_t)$.

		Let $F$ denote a graph that is not Ramsey for $(K_{1,s},K_t\cdot K_2)$, and let $c$ be a $(K_{1,s},K_t\cdot K_2)$-free coloring of $E(F)$ that minimizes the number of blue copies of $K_t$ among all such colorings. We claim that $c$ has no blue copies of $K_t$ and hence $F\not\to (K_{1,s},K_t)$.
		
		For a contradiction, assume that there exists a blue copy of $K_t$ under $c$. Note that the blue copies of $K_t$ must be pairwise disjoint and that there are no blue edges leaving any of these copies, that is, the copies of $K_t$ form isolated components in the blue subgraph of $F$ under $c$. 
		Then we greedily choose an alternating walk $W$ as follows, where we note that $W$ may visit some vertices twice but uses each edge at most once: We start with an arbitrary edge $e$ belonging to a blue copy of $K_t$ and extend $W$ in both directions.
		Now, all edges of $F$ incident to $e$ that are not in this copy of $K_t$ are red.
		From each endpoint of $e$ the walk then follows one such red edge (if it exists).
		If (in either direction) the walk then reaches a previously unvisited blue copy of $K_t$, the walk follows an arbitrary blue edge in this copy and then proceeds with some previously unvisited red edge as before.
		This procedure stops in some direction if either an endpoint of the last blue edge does not have any unvisited incident red edge, or if the last vertex of the walk is contained in a blue copy of $K_t$ that already has an edge in the walk, see Figure~\ref{fig:recolouringwalk}, or if the last vertex is in no copy of $K_t$. This implies that in the latter two cases the last edge in this direction is red. Observe that, as soon as $W$ repeats a vertex, it must stop in that direction at that vertex: if $W$ repeats a vertex before it stops in some direction, then this vertex has degree at least four in $W$ and is therefore incident to at least two blue edges in $W$. But this is not possible since the blue copies of $K_t$ are disjoint and $W$ traverses at most one edge from each such copy. In other words, except possibly for the last step in each direction, $W$ must be a red/blue-alternating path. Therefore, each vertex of $W$ that is not an endpoint is incident to exactly one red and one blue edge; the endpoints may be incident to multiple red edges but again at most one blue edge.

		\begin{figure}
			\centering
			\includegraphics{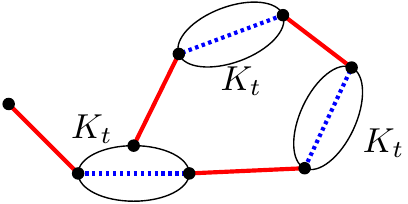}
			\caption{An alternating walk (with solid red edges and dotted blue edges) chosen according to the proof of Theorem~\ref{thm:equivalentpairs}\ref{thm:StarClique}.}
			\label{fig:recolouringwalk}	\end{figure}
		
		We now switch the colors of the edges in $W$ and call the new coloring $\tilde{c}$. 
		We claim that the coloring $\tilde{c}$ contains
		\begin{enumerate}
			\item no red copy of $K_{1,s}$, 
			\item fewer blue copies of $K_t$ than $c$, and
			\item no blue copy of $K_t\cdot K_2$. 
		\end{enumerate}
		By the definition of $c$, this already leads to the desired contradiction. 		
		
		To prove (1), observe that the switch does not change the number of red edges on the inner vertices of $W$. Now consider an endpoint $v$ of $W$.
		The number of red edges incident to $v$ in $W$ decreases if the corresponding first or last edge of $W$ is red.
		If the corresponding edge is blue, then there are no red edges under $c$ incident to $v$ by the construction of $W$. 
		Hence $v$ is incident to exactly one red edge under $\tilde{c}$. Therefore, there is no red copy of $K_{1,s}$ under $\tilde{c}$.
		
		\smallskip
		We prove (2) now. First observe that $W$ contains at least one edge belonging to a copy of $K_t$ that is blue under $c$. So, after switching colors, this copy of $K_t$ contains a red edge. Therefore, the only way for (2) to fail is that we create a new blue copy of $K_t$ by switching the colors along $W$. So, consider any edge $uv$ in $W$ whose color switched from red to blue and such that $uv$ is contained in a copy $K$ of $K_t$. We aim to show that $K$ is not monochromatic blue under $\tilde{c}$.
		To do so, choose an arbitrary vertex $x\in V(K)\setminus \{u,v\}$. What we will see is that either $ux$ or $vx$ is red under $\tilde{c}$, which will prove the claim. Assume that the statement is false; we consider three cases depending on the original colors of $ux$ and $vx$.
		
		\noindent\textbf{Case 1}: Assume $ux$ and $vx$ are both red under $c$. Then both of these edges and $uv$ must have switched colors and hence belong to $W$. This however contradicts the fact that at most two vertices are incident to two red edges of $W$ under $c$.
		
		\noindent\textbf{Case 2}: Assume $ux$ and $vx$ are both blue under $c$. By the construction of $W$, at least one of $u$ and $v$, say $u$, must be contained in a copy $K'$ of $K_t$ entirely colored blue under $c$. But then $K'$ together with the edge $ux$ (if $x\notin V(K')$) or the edge $vx$ (if $x\in V(K')$) forms a blue copy of $K_t\cdot K_2$ under $c$, a contradiction to the choice of $c$.
		
		\noindent\textbf{Case 3}: Assume $ux$ is red and $vx$ is blue under $c$ (the case where $ux$ is blue and $vx$ is red is similar). Then $vx\notin W$, and both $ux$ and $uv$ have switched colors, i.e., $ux,uv\in W$. Therefore, $u$ is incident to two red edges of $W$ under the coloring $c$, and hence must be an endpoint of $W$. Moreover, we see that there must exists a copy $K'$ of $K_t$ entirely colored blue under $c$ which contains both $v$ and $x$. Indeed, by the construction of $W$, at least one of the vertices $v$ and $x$ must be contained in such a copy $K'$, and the other vertex must also belong to this copy since $vx$ is blue and $c$ does not contain a blue copy of $K_t\cdot K_2$.
		Now, we may assume that in our procedure the edge $ux$ is added to $W$ after the edge $uv$ (otherwise exchange the role of the two edges). Then, since $vx\notin W$, it follows that, before the procedure adds $ux$ to $W$, the walk meets the clique $K'$ already a second time when it reaches $x$. But then the procedure stops at $x$ and hence does not add $ux$ to $W$, a contradiction. 
		
		\smallskip
		It remains to check property (3). Assume that there is a blue copy of $K_t\cdot K_2$ in $\tilde{c}$, with $t$-clique $K$ and
		pendent edge $f$. As we have already seen, the switching of colors does not create new blue copies of $K_t$. Hence, $K$ must be blue under $c$ and all edges intersecting $K$ in exactly one vertex must be red under $c$, since $c$ does not contain a monochromatic copy of $K_t\cdot K_2$. But this means that $f$ must be red in $c$ and hence $f\in W$. By the definition of $W$, some edge of $K$ must be added to it, and the color of this edge is then switched from blue to red, a contradiction.
	\end{proof}

	To prove Theorem~\ref{thm:equivalentpairs}\ref{thm:StarCaterpillar} we make use of the following definition.
	We call a graph $G$ \emph{$k$-woven} if, 
	for each graph $F$ that contains an edge $uv$ which is contained in all copies of $G$ in $F$, there is a set $Y_{uv}\subseteq E(F)\setminus\{uv\}$ such that the following holds: 
	$Y_{uv}$ consists of at most $k$ edges incident to $u$ and at most $k$ edges incident to $v$, and each copy of $G$ in $F$ contains an edge from $Y_{uv}$.
	In other words, $Y_{uv}$ is a set of edges of size at most $2k$ whose removal yields a graph with no copies of $G$ that still contains the edge $uv$. 
	As a simple example, it is not difficult to check that stars with at least two edges are $1$-woven.
	Indeed, if a graph $F$ has an edge $uv$ that is contained in each copy of some star $K_{1,s}$ in $F$, then $u$ and $v$ are of degree at most $s$ in $F$ (and all other vertices are of degree at most $s-1$).
	So any set $Y_{uv}$ consisting of one edge incident to $u$ and one edge incident to $v$ in $G-uv$ satisfies the condition sated above, that is, each copy of $K_{1,s}$ in $F$ contains an edge from $Y_{uv}$.	
	We begin by demonstrating the utility of $k$-woven graphs in Proposition~\ref{prop:woven} below. We will then prove Theorem~\ref{thm:equivalentpairs}\ref{thm:StarCaterpillar} by showing that suitable caterpillars are $k$-woven for appropriately chosen $k$. For any pair of graphs $(G,H)$, we write $r(G,H)$ for the Ramsey number of $(G,H)$, i.e., the smallest integer $n$ such that $K_n \rightarrow (G,H)$.

	\begin{prop}\label{prop:woven}
		Let $G$ be a $k$-woven graph, let $a\geq 1$ and $b\geq 2$ be integers, and let $r=r(G,K_{b-1})$.
		If $t\geq 4k + 2 (r+(a-1)(b-1)) + (a-1)$, then $(G,K_t)\sim(G,K_t\cdot aK_b)$.
	\end{prop}
	\begin{proof}
		Clearly, each Ramsey graph for $(G,K_t\cdot aK_b)$ is also a Ramsey graph for $(G,K_t)$.
		So consider a graph $F$ with $F\not \to (G,K_t\cdot aK_b)$.
		We shall show that $F\not \to (G,K_t)$.
		Let $\varphi_1$ denote a 
		$(G,K_t\cdot aK_b)$-free coloring of $E(F)$.
		For each blue copy $K$ of $K_t$ in $F$, let $U_K\subseteq V(K)$ denote the set of vertices $u$ in $K$ such that there are at least $r+(a-1)(b-1)$ blue edges between $u$ and $F-K$ whose endpoints induce a complete graph in $F-K$.
		See Figure~\ref{fig:wovenEquivalence} (left).
		
		\begin{figure}
			\centering
			\includegraphics{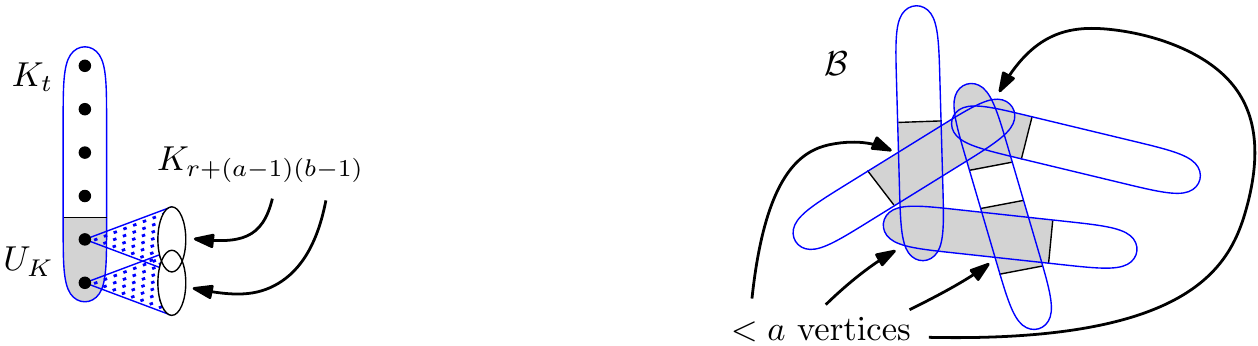}
			\caption{Left: The set $U_K$ (gray background) in a copy $K$ of $K_t$. Right: A set $\mathcal{B}$ of blue copies of $K_t$ under $\varphi_1$ with pairwise intersection of size less than $a$. The intersections are contained in the respective sets $U_K$ by Claim~\ref{clm:intersections}.}
			\label{fig:wovenEquivalence}
		\end{figure}
		
		Let $\mathcal{B}$ denote a maximal set of blue copies of $K_t$ in $F$ such that any two copies of $K_t$ in $\mathcal{B}$ intersect in fewer than $a$ vertices.
		See Figure~\ref{fig:wovenEquivalence} (right).
		We first make several general observations.
		
		\begin{clm}\label{clm:intersections}
			For any two copies $K$, $K'\in\mathcal{B}$, we have $V(K)\cap V(K') \subseteq U_K\cap U_{K'}$ and hence $(V(K)\setminus U_K)\cap (V(K')\setminus U_{K'})=\emptyset$.
		\end{clm}
		\begin{proof}
			For each vertex $u\in V(K)\cap V(K')$ the number of blue edges between $u$ and $K-K'$, as well as between $u$ and $K'-K$, is at least $t-(a-1)\geq 4k + 2 (r+(a-1)(b-1))  \geq r+(a-1)(b-1)$. Since these blue neighborhoods induce a complete graph, we have $u\in U_K$ and $u\in U_{K'}$.
		\end{proof}
		
		\begin{clm}\label{clm:leftover}
			For every $K\in \mathcal{B}$, we have $\lvert U_{K}\rvert \leq a-1$ and hence $\lvert V(K)\setminus U_{K}\rvert \geq 2$.
		\end{clm}
		\begin{proof}
			We first argue that $\lvert U_{K}\rvert \leq a-1$.
			This holds since, under $\varphi_1$, each vertex in $U_{K}$ has a blue neighborhood of size at least $r+(a-1)(b-1)$ in $F-K$ that induces a complete graph.
			As  there are no red copies of $G$ under $\varphi_1$, by the definition of Ramsey number, we iteratively find $\lvert U_{K}\rvert$ vertex-disjoint blue copies of $K_{b-1}$ in the blue neighborhood of $U_{K}$, one for up to $\min\{a,\lvert U_{K}\rvert\}$ vertices in $U_{K}$.
			So $\lvert U_{K}\rvert \leq a-1$, as there is no blue copy of $K_t\cdot aK_b$ under $\varphi_1$.
			This shows that $\lvert V(K)\setminus U_{K}\rvert \geq t-a+1 \geq 2$, as required.
		\end{proof}
		
		We shall now recolor some edges contained in or incident to cliques in $\mathcal{B}$ to obtain an 
		$(G,K_t)$-free coloring of $E(F)$.
		Let $M$ denote a largest matching in the graph induced by  $\cup_{K\in\mathcal{B}}(V(K)\setminus U_K)$.
		That is, $M$ consists of a largest matching from each of the cliques $K-U_K$, which are vertex-disjoint by Claim~\ref{clm:intersections} (see Figure~\ref{fig:wovenEquivalence2}).
		Let $\varphi_2$ denote the coloring obtained from $\varphi_1$ by switching the color of each edge in $M$ from blue to red.
		Each red copy of $G$ under $\varphi_2$ contains an edge from $M$.
		Let $u_1v_1,\ldots,u_{\lvert M\rvert}v_{\lvert M\rvert}$ denote the edges of $M$ in an arbitrary order.		
		We shall use the fact that $G$ is $k$-woven to find sets $Y_1,\ldots,Y_{\lvert M\rvert}\subseteq E(F)\setminus M$ such that each set $Y_i$ consists of at most $k$ red edges incident to $u_i$ and at most $k$ red edges incident to $v_i$ and such that each red copy of $G$ in $F$ under $\varphi_2$ contains an edge from $Y_i$ for some $i \in [\lvert M\rvert]$.
		To do so consider the subgraph $F_1$ of $F$ formed by all red copies of $G$ under $\varphi_2$ containing $u_1v_1$ and not containing $u_jv_j$ for $j>1$.
		Then $F_1-u_1v_1$ contains no copy of $G$, and hence, since $G$ is $k$-woven, there is a desired set $Y_1\subseteq E(F_1)\setminus\{u_1v_1\}$.
		For $i>1$ we proceed iteratively.
		Having chosen $Y_1,\ldots,Y_{i-1}$, let $F_i$ denote the subgraph of $F$ formed by all red copies of $G$ under $\varphi_2$ containing $u_iv_i$, not containing any edge from $Y_j$ for $j<i$, and not containing $u_jv_j$ for $j>i$.
		Then $F_i-u_iv_i$ contains no copy of $G$.
		Hence, since $G$ is $k$-woven, there is a desired set $Y_i\subseteq E(F_i)\setminus\{u_iv_i\}$.

		\begin{figure}
			\centering
			\includegraphics{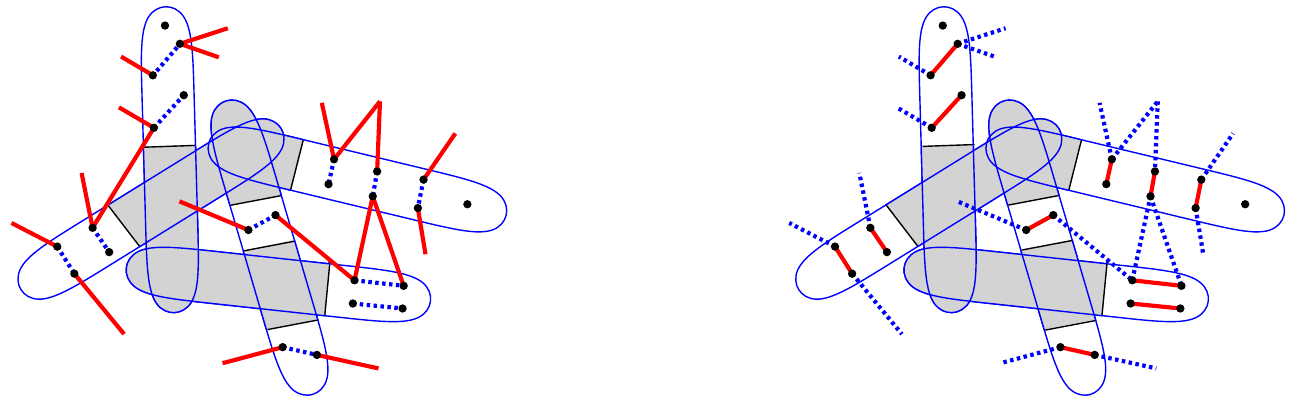}
			\caption{Left: A largest blue (dotted) matching $M$ in $\cup_{K\in\mathcal{B}}(V(K)\setminus U_K)$ with some pendent red (solid) edges under $\phi_1$. Right: The final coloring $\varphi_3$ is obtained by switching colors in $M$ (from blue to red) and at most $k$ further incident edges at each vertex in $M$ (from red to blue).}
			\label{fig:wovenEquivalence2}
		\end{figure}
		
		Now let $\varphi_3$ denote the coloring obtained from $\varphi_2$ by switching the color of each edge in $\cup_{1\leq i\leq \lvert M\rvert}Y_i$ from red to blue (see Figure~\ref{fig:wovenEquivalence2}).
		Then there are no red copies of $G$ under $\varphi_3$.
		Indeed, each red copy $G'$ of $G$ under $\varphi_2$ contains an edge from $Y_i$ for some $i$.
		We shall prove that there are no blue copies of $K_t$ under $\varphi_3$.
		Let $K'$ denote a copy of $K_t$ in $F$.
		
		First suppose that $K'\in\mathcal{B}$. By Claim~\ref{clm:leftover}, we have  $\lvert V(K')\setminus U_{K'}\rvert \geq 2$, and hence $K'$ contains a red edge under $\varphi_3$ from $E(K')\cap E(M)$.
		
		We may assume then that $K'\not\in\mathcal{B}$.
		If for each $K\in\mathcal{B}$ we have $V(K)\cap V(K')\subseteq U_K$, then $\lvert V(K)\cap V(K')\rvert < a$ by Claim~\ref{clm:leftover}.
		By the maximality of $\mathcal{B}$, $K'$ contains a red edge under $\varphi_1$. This edge is red under $\varphi_3$, since only edges incident to $M$ switched colors from red to blue and the edges in $K'$ are not incident to $M$ (as $V(K)\cap V(K')\subseteq U_K$ for each $K\in\mathcal{B}$ here). So $K'$ is not blue in this case.

		If there is $K\in\mathcal{B}$ with $\lvert V(K)\cap V(K')\rvert > \lceil\frac{t-\lvert U_K\rvert}{2}\rceil + \lvert U_K\rvert$, then $K'$ contains an edge from $M\cap K$.
		This edge is red under $\varphi_3$ and so $K'$ is not blue.

		If neither of the two previous cases holds, then let $V=\cup_{K\in\mathcal{B}}(V(K')\cap (V(K)\setminus U_K))$.
		By assumption $\lvert V\rvert\geq 1$ (since we are not in the first case).
		Each vertex $v\in V$ is contained in exactly one $K\in\mathcal{B}$ and, since we are not in the second case, the number of edges between $v$ and $K'-K$ is at least $t-\lceil\frac{t-\lvert U_K\rvert}{2}\rceil - \lvert U_K\rvert = \lfloor\frac{t-\lvert U_K\rvert}{2}\rfloor$.
		Since $v\not\in U_K$, fewer than $r+(a-1)(b-1)$ of those edges are colored blue under $\varphi_1$ (as their endpoints induce a complete subgraph of $K'$).
		Together, this means that the number of red edges under $\varphi_1$ incident to $v$ in $K'$ is at least $\lfloor\frac{t-\lvert U_K\rvert}{2}\rfloor-r-(a-1)(b-1)+1 \geq 2k+1$, using the fact that $\lvert U_K\rvert\leq a-1$ by Claim~\ref{clm:leftover}.
		In total there are at least $(2k+1)\lvert V\rvert / 2>k\lvert V\rvert$ red edges in $K'$ under $\varphi_1$.
		To obtain $\varphi_3$, for each $v\in V$, at most $k$ incident edges were chosen to switch colors from red to blue (in total more edges incident to $v$ than those $k$ might have switched colors due to other vertices in $V$).
		So at most $k\lvert V\rvert$ edges in $K'$ switched their color from red to blue.
		This shows that at least one edge in $K'$ is red under $\varphi_3$.

		Altogether there are no red copies of $G$ and no blue copies of $K_t$ under $\varphi_3$ and hence $F\not \to (G,K_t)$.
	\end{proof}
	
	\begin{proof}[Proof of Theorem~\ref{thm:equivalentpairs}\ref{thm:StarCaterpillar}]
		
		By Proposition~\ref{prop:woven}, it suffices to show that stars and suitable caterpillars are $k$-woven for some $k$. As mentioned above, it is not difficult to check that  stars with at least two edges are $1$-woven. We now focus on caterpillars, and claim that every $s$-suitable caterpillar is $2(s+1)^2$-woven.
		
		Let $T$ be an $s$-suitable caterpillar, that is, $T$ consists of a path $abc$ and $s$ leaves adjacent to $a$, $s$ leaves adjacent to $c$, and $s'<s$ leaves adjacent to $b$.
		We shall prove that $T$ is $k$-woven for $k=2(s+1)^2$.
		Let $F$ be a graph with an edge $uv$ that is contained in all copies of $T$ in $F$, and let $F'=F-uv$.
		We need to find a set $Y_{uv}\subseteq E(F')$ consisting of at most $k$ edges incident to $u$ and at most $k$ edges incident to $v$ such that $Y_{uv}$ contains an edge from each copy of $T$ in $F$.
		
		First suppose that there is a copy $T_0$ of $T$ in $F$ in which $u$ is a leaf.
		The neighbor of $u$ in $T_0$ is $v$, since $T_0$ contains $uv$ by assumption.
		Then $v$ is of degree at most $\lvert V(T)\rvert-2=1+2s+s'\leq k$ in $F'$, since otherwise $uv$ can be replaced in $T_0$ by some edge $vw$ in $F'$ to form a copy of $T$ entirely in $F'$, which does not exist by assumption. 
		Let $Y_{v}$ consist of all edges in $F'$ incident to $v$.
		If each copy of $T$ in $F$ contains an edge from $Y_v$, then we can choose $Y_{uv}=Y_v$ as our desired 
		edge set.
		Otherwise, there is a copy of $T$ in $F$ containing no edges from $Y_v$.
		In such a copy of $T$ the vertex $v$ is a leaf, since $uv$ is the only edge incident to $v$ not in $Y_v$.
		Similarly as above, $u$ is of degree at most $k$ in $F'$, and hence we can choose $Y_{uv}$ to consist of all edges incident to $u$ and all edges incident to $v$ in $F'$.
		
		It remains to consider the case where neither $u$ nor $v$ is a leaf in any copy of $T$ in $F$.
		By the symmetry of $T$, we may assume that in each copy of $T$ in $F$ the edge $ab$ corresponds to $uv$, where $a$ corresponds to either $u$ or $v$.
		Let $N_u$ and $N_v$ denote the set of neighbors of $u$ and $v$ in $F'$, respectively, that are of degree at least $s+1$ in $F'$ (see Figure~\ref{fig:wovenCaterpillar} left).
		In each copy of $T$ in $F$ the edge $bc$ corresponds to an edge $uw$ with $w\in N_u$ or an edge $vw$ with $w\in N_v$.
		In particular choosing $Y_{uv}=\{uw\colon w\in N_u\}\cup \{vw\colon w\in N_v\}$ yields the desired edge set, provided that $\lvert N_u\rvert$, $\lvert N_v\rvert\leq k$.
		In the following, we prove $\lvert N_u\rvert\leq k$. By symmetry the same bound holds for $\lvert N_v\rvert$.

		For a contradiction, assume that $\lvert N_u\rvert \geq k+1$, which in particular implies that $u$ is of degree at least $k+1$ in $F'$.
		We shall prove that there is a copy of $T$ in $F'$ under this assumption.
		For each $w\in N_u$, choose a star in $F'$ 
		with center vertex $w$ and exactly $s$ leaves not containing $u$,
		and let $\mathcal{S}$ denote the set of all chosen stars.
		For any two such stars $S$, $S'\in\mathcal{S}$ there are at least $k+1-2(s+1)\geq s > s'$ neighbors of $u$ in $F'$ not contained in $V(S)\cup V(S')$.
		Since $F'$ does not contain a copy of $T$, the stars $S$ and $S'$ must intersect in some vertex, which could be the center of one of the stars but not of both (see Figure~\ref{fig:wovenCaterpillar} middle).
		
		\begin{figure}	
			\centering	
			\includegraphics{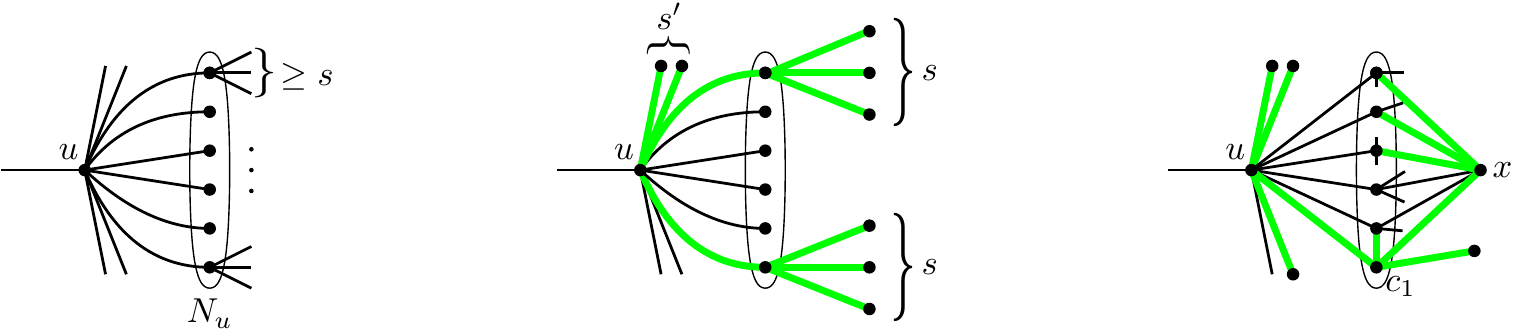}
			\caption{Left: The vertex $u$ and the set $N_u$ of at least $k+1$ neighbors of $u$ of degree at least $s+1$ each. Middle: If two vertices in $N_u$ have only $u$ as a common neighbor, then there is copy of $T$ (green/thick gray). Right: Otherwise, there is a vertex $x$ that is a leaf in at least $2s+1$ copies of $K_{1,s}$ centered in $N_u$ and there is a copy of $T$ (here a $3$-suitable caterpillar) as well (green/thick gray).}
			\label{fig:wovenCaterpillar}
		\end{figure}

		Now, consider some fixed star $S\in\mathcal{S}$.
		By the pigeonhole principle, there is a vertex $x$ in $V(S)$ that is contained in at least 
		$\lvert\mathcal{S}\setminus\{S\}\rvert/\lvert V(S)\rvert = (\lvert N_u\rvert-1)/(s+1) \geq \frac{k}{s+1}=2(s+1)$ of the stars in $\mathcal{S}$.
		It may happen that $x$ is the center vertex of one such star, but in any case there is a family of $2s+1$ stars in $\mathcal{S}$ that have $x$ as a leaf.
		Let $X=\{c_1,c_2,\ldots,c_{2s+1}\}\subseteq N_u$ denote the set of their centers.
		Then we find a copy of $T$ in $F'$ as follows:
		Let $X_1$ denote a set of $s'$ neighbors of $c_1$ distinct from $u$ and $x$ in $F'$, which exists since $s'<s$ and $c_1\in X$ is of degree $s+1$ in $F'$.
		Let $X_2$ denote a set of $s$ neighbors of $x$ in $X$ disjoint from $X_1\cup\{c_1\}$, which exists since $|X|=2s+1\geq s+s'+2$.
		Finally, let $X_3$ denote a set of $s$ neighbors of $u$ in $F'$ disjoint from $X_1\cup X_2\cup \{x,c_1\}$, which exists since the degree of $u$ in $F'$ is at least $k+1=2(s+1)^2+1\geq 2s+s'+2$.
		Then the path $uc_1x$ together with the vertices in $X_1$, $X_2$, and $X_3$ induces a copy of $T$ in $F'$, a contradiction (see Figure~\ref{fig:wovenCaterpillar} right).
		This shows that $\lvert N_u\rvert \leq k$ and, by symmetry,  $\lvert N_v\rvert \leq k$. Hence, $Y_{uv}=\{uw\colon w\in N_u\}\cup \{vw\colon w\in N_v\}$ is the desired edge set.
	\end{proof}

	\section{Non-equivalence results for trees and cliques}\label{sec:nonequivalent}
	
	In this section, we prove each part of Theorem~\ref{thm:non-equivalentpairs} in turn.
	When constructing appropriate distinguishing graphs in our proofs, we will often combine several smaller graphs, which we call building blocks, by identifying some of their vertices or edges. We will assume that, except for the specified intersections, all of these building blocks are disjoint from one another.
	
	\subsection{Proof of Theorem~\ref{thm:non-equivalentpairs}\ref{thm:diameterTrees}}	
	
	Recall that the diameter of a tree $T$, denoted $\text{diam}(T)$, is the length of its longest path.
	Our construction gives $(T,K_t)\not\sim (T,H)$ for each tree $T$ from the following slightly larger class $\mathcal{T}'$ consisting of all trees $T$ that either have odd diameter, or have even diameter and additionally satisfy the following:
	\begin{itemize}
		\item The central vertex of $T$ has at most one neighbor of degree at least three that is contained in a longest path in $T$.
		\item If $T$ is of diameter four, the central vertex is of degree at least three.
	\end{itemize}
	See Figure~\ref{fig:diameterTrees} (left) for an illustration.
	
	Theorem~\ref{thm:non-equivalentpairs}\ref{thm:diameterTrees} is clearly a direct consequence of Theorem~\ref{thm:non-equivalent_trees} below.
	
	\begin{thm}~\label{thm:non-equivalent_trees}
		Let $t\geq 3$, let $H$ be a connected graph, and let $T\in\mathcal{T'}$.
		Then $(T,K_t)\not\sim (T,H)$.
	\end{thm}
	\begin{proof}
		Consider a tree $T\in\mathcal{T'}$.
		If $\omega(H)\neq t$, then $(T,K_t)\not\sim (T,H)$ by Theorem~\ref{thm:NRcliquenumber}.
		So we assume $K_t\subsetneq H$ for the remainder of the proof.
		We will construct a Ramsey graph for $(T,K_t)$ that is not Ramsey for $(T,K_t\cdot K_2)$ and hence not Ramsey for $(T,H)$.
		The construction differs slightly depending on the parity of the diameter of $T$. We begin by introducing a useful gadget graph.
		\smallskip
		
		Throughout the proof, we let $U_{k,i}$ denote the rooted tree in which every leaf is at distance $i$ from the root and every vertex that is not a leaf has exactly $k$ children.
		Here, the distance between two vertices $x$ and $y$ is the length of a shortest path that has $x$ and $y$ as its endpoints.
		Note that $U_{k,i}$ contains every tree of diameter at most $2i$ and maximum degree at most $k$.
		
		Let $d$ denote the maximum degree of $T$.
		Let $\Gamma$ be a Ramsey graph for $(T,K_{t-1})$ that does not contain a copy of $K_t$,
		which exists by Theorem~\ref{thm:NRcliquenumber}. Write $k=d\lvert V(\Gamma)\rvert$.
		For a positive integer $i$, let $\Lambda_i=\Lambda_i(T,\Gamma)$ denote the graph obtained from a copy of $U_{k,i}$ by adding edges so that, for each non-leaf vertex of $U_{k,i}$, its set of children induces $d$ vertex-disjoint copies of $\Gamma$. 
		We refer to the root of $U_{k,i}$ as the root of $\Lambda_i$.
		Let $\Phi_i=\Phi_i(\Lambda_i)$  be the red/blue-coloring that assigns red to all edges in $U_{k,i}$ and blue to all the other edges, see Figure~\ref{fig:diameterTrees} (middle) for an illustration. Observe that, if $i <\text{diam}(T)$, then $\Phi_i$ 
		is a $(T,K_t)$-free coloring of $E(\Lambda_i)$. We have the following Ramsey property of $\Lambda_i$.
		
		\begin{clm}\label{obs:LambdaRamsey}
			Every red/blue-coloring of $E(\Lambda_i)$ yields a 
			red copy of $T$, a blue copy of $K_t$, or a red copy of $U_{d,i}$ whose root is the root of $\Lambda_i$.
		\end{clm}
		\begin{proof}
			To see why this is true, consider an arbitrary $2$-coloring of $E(\Lambda_i)$ with no red copy of $T$. 
			Then each copy of $\Gamma$ contains a blue copy of $K_{t-1}$.
			If some non-leaf vertex in $U_{k,i}$ has only blue edges to one of the copies of $\Gamma$ formed by its children, then there is a blue copy of $K_t$. Otherwise, every such vertex has a red edge to each of the $d$ copies of $\Gamma$ formed by its children, yielding a copy of $U_{d,i}$ as required.
		\end{proof}
		
		\begin{figure}
			\centering
			\includegraphics{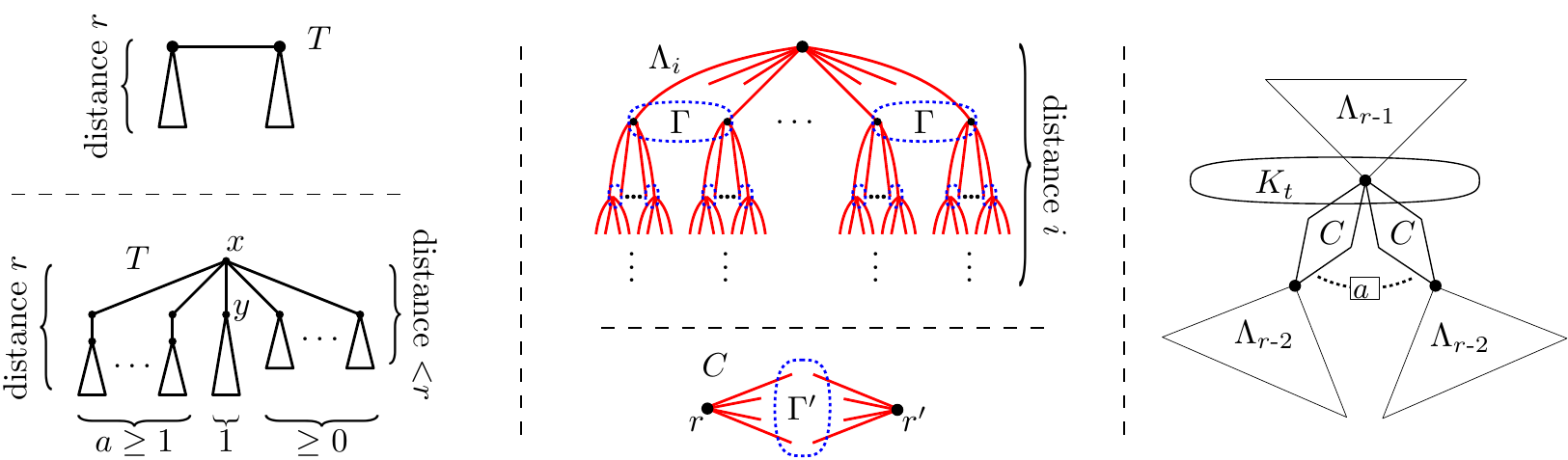}
			
			\caption{Left: An odd diameter tree and an even diameter tree from the class $\mathcal{T'}$. Middle: Graphs $\Lambda_i$ and $C$ with the respective $(T,K_t\cot K_2)$-free colorings. Right: A graph $F$ with $F\to (T,K_t)$ and $F\not\to (T,K_t\cdot K_2)$ in case $T$ is of even diameter.}
			\label{fig:diameterTrees}
		\end{figure}
		
		\medskip
		First consider the case where $T$ is of diameter $2r+1$ for some integer $r$.
		We construct a graph $F$ as follows:
		Start with a copy $K$ of $K_t$.
		For each vertex $u$ of $K$, add a copy of $\Lambda_r$ rooted at $u$ so that  the copies of $\Lambda_r$ are pairwise disjoint. We claim that $F\to (T,K_t)$ and $F\not\to(T,K_t\cdot K_2)$.

		To prove the first claim, we consider an arbitrary $2$-coloring of $E(F)$ with no red copy of $T$. By Claim~\ref{obs:LambdaRamsey}, some copy of $\Lambda_r$ contains a blue copy of $K_t$ or  each vertex of $K$ is the root of a red copy of $U_{d,r}$. If we find a blue copy of $K_t$, we are done, and hence we may assume that the latter happens for every vertex of $K$. If there is a red edge in $K$, then this edge and the red copies of $U_{d,r}$ rooted at its endpoints form a graph which contains a red copy of $T$.
		Otherwise, all edges of $K$ are colored blue, yielding a blue copy of $K_t$.
		This shows  $F\to (T,K_t)$.
		
		To see that $F\not\to(T,K_t\cdot K_2)$, color all edges of $K$ blue and give all copies of $\Lambda_r$ the coloring $\Phi_r$.
		Then $K$ is the only blue copy of $K_t$ and it cannot be extended to a copy of $K_t\cdot K_2$, as all edges leaving $K$ are colored red and all the other blue edges form vertex-disjoint copies of $\Gamma$, which was chosen such that $K_t \nsubseteq \Gamma$.
		The red edges form vertex-disjoint trees of diameter $2r <$ diam$(T)$.
		Hence, there is no red copy of $T$ and no blue copy of $K_t\cdot K_2$ and so $F\not\to(T,K_t\cdot K_2)$.
		
		\medskip
		Now consider the case where $T$ is of diameter $2r$ for some integer $r$.
		In this case the assumptions on $\mathcal{T}'$ imply that at most one neighbor of the central vertex is of degree at least three and is contained in a longest path. Further, if the diameter is exactly four, the central vertex of $T$ is of degree at least three.
		Let $x$ denote the central vertex of $T$, let $y$ denote a neighbor of $x$ in $T$ that is of largest degree among all neighbors of $x$ contained in a longest path in $T$, and let $a$ denote the number of all other neighbors of $x$ contained in a longest path in $T$ (see Figure~\ref{fig:diameterTrees} (left) for an illustration).
		By the assumption on the structure of $T$, all neighbors of $x$ counted by $a$ are of degree exactly two in $T$. As in the previous case, we will use the graphs $\Lambda_i$ as building blocks. We now define the second type of building block that we will use in the construction.
		\smallskip
		
		Let $J$ denote a graph containing no copy of $K_t$ such that, for any $2$-coloring of the \emph{vertices} of $J$, there is a vertex-monochromatic copy of $K_{t-1}$.
		Such a graph exists by~\cite{folkman1970graphs}.
		Let $\Gamma'$ be a Ramsey graph for $(T,J)$ not containing a copy of $K_t$, which exists by Theorem~\ref{thm:NRcliquenumber}, and let $k'=\lvert V(\Gamma')\rvert$.
		Let $C$ denote the graph obtained from a copy of $\Gamma'$ by adding two non-adjacent vertices $r$ and $r'$ and a complete bipartite graph between these two vertices and the vertices of the copy of $\Gamma'$. For convenience we call $r$ the root of $C$.
		See Figure~\ref{fig:diameterTrees} (middle) for an illustration.

		\begin{clm}\label{obs:CRamsey}
			In every red/blue-coloring of $E(C)$, there exists a red copy of $T$, a blue copy of $K_t$, or a red path $rvr'$ for some $v\in V(\Gamma')$.
		\end{clm}
		\begin{proof}
			To see why this is true, consider a red/blue-coloring of $E(C)$ with no red copy of $T$. Then the copy of $\Gamma'$ in $C$ contains a blue copy $J'$ of $J$.
			In particular, each copy of $K_{t-1}$ in $J'$ is blue. Moreover, either each copy of $K_{t-1}$ has a red edge going to each of $r$ and $r'$, or there is a blue copy of $K_t$. In the latter case, we are done, so assume the former. 
			Consider an auxiliary vertex coloring of $J'$ obtained by coloring each vertex $v$ in $J'$ with the color of the edge $rv$. Since there cannot be a vertex-monochromatic blue copy of  $K_{t-1}$, there is a vertex-monochromatic red copy of $K_{t-1}$.
			We assume that there is a red edge between this copy of $K_{t-1}$ and $r'$, and hence we find a red path $rvr'$, where $v\in V(J')$.
		\end{proof}
		
		We now construct a graph $F'$ as follows: Start with a copy $K'$ of $K_t$. For each vertex $u$ of $K'$, add $a$ copies of $C$ and a copy of $\Lambda_{r-1}$ all rooted at $u$.
		Further, for each copy of $C$, add a copy of $\Lambda_{r-2}$ rooted at the copy of $r'$. 
		See Figure~\ref{fig:diameterTrees} (right) for an illustration. We claim that $F'\to (T,K_t)$ and $F'\not\to(T,K_t\cdot K_2)$.

		To prove the first claim we consider an arbitrary $2$-coloring of $E(F')$ with no red copy of $T$. By Claim~\ref{obs:LambdaRamsey}, either there is a blue copy of $K_t$ in some copy of $\Lambda_{r-1}$ or $\Lambda_{r-2}$ or the root of each copy of $\Lambda_{r-1}$ or $\Lambda_{r-2}$ is the root of a red copy of $U_{d,r-1}$ or $U_{d,r-2}$, respectively. Assume the latter is true, since otherwise we are done. By Claim~\ref{obs:CRamsey}, either there exists a blue copy of $K_t$ in some copy of $C$ or each copy of $C$ in $F$ contains a red path of length two connecting the copies of $r$ and $r'$. Again, we may assume that we are in the latter case.  For each copy of $C$, the copy of $r'$ is the root of a copy of $\Lambda_{r-2}$. Now, every vertex of $K'$ is the root of $a$ copies of $C$ and a copy of $\Lambda_{r-1}$. If there is a red edge $e$ in $K$, then there is a red copy of $T$ formed by $e$ and subtrees of the red trees rooted at its endpoints, with $e$ playing the role of the edge $xy$ in $T$. 
		Otherwise all edges of $K$ are colored blue, proving the claim.

		To show that $F'\not\to(T,K_t\cdot K_2)$ we color the edges of $F'$ as follows:
		All edges of $K'$ are colored blue and all mentioned copies of $\Lambda_i$ ($i\in\{r-1,r-2\}$) are colored according to the coloring $\Phi_i$, as defined earlier.
		For each mentioned copy of $C$, all edges in the  copy of $\Gamma'$ are blue and all other edges red.
		Then $K'$ is the only blue copy of $K_t$, as all other blue edges form vertex-disjoint copies of $\Gamma$ or $\Gamma'$ and these graphs do not contain copies of $K_t$.
		Moreover, $K'$ has only red incident edges.
		So there is no blue copy of $K_t\cdot K_2$.
		Now consider the red subgraph of $F'$, and recall that the central vertex of $T$ has $a+1$ neighbors contained in paths of length $2r$.
		If $r>2$, then each longest red path in $F'$ has $2r$ edges and the middle vertex of each such path is in $K'$.
		In particular, the central vertex of each red tree of diameter $2r$ is in $K'$.
		But for each vertex $u\in V(K')$ there are at most $a$ red paths of length $2r$ that meet at $u$ and are otherwise pairwise vertex-disjoint, so there is no red copy of $T$.
		If $r=2$, then for the same reason there is no red copy of $T$ rooted in $K'$.
		In this case there are also red paths of length $2r=4$ whose central vertex is in the neighborhood of $K'$.
		However, these vertices are of degree at most two in the red subgraph and, by assumption, the root of $T$ is of degree at least three in this case.
		This shows that $F\not\to (T,K_t\cdot K_2)$.
	\end{proof}

	\subsection{Proof of Theorem~\ref{thm:non-equivalentpairs}\ref{thm:noexternalcycle}}  We first introduce some definitions. A \emph{cycle of length $s$} in a hypergraph $\mathcal{H}$ is a sequence $e_1,v_1,e_2,v_2\dots,e_s,v_s$ of distinct hyperedges and vertices such that $v_i\in e_i\cap e_{i+1}$ for all $1\leq i\leq s$ where $e_{s+1}=e_1$. The \emph{girth} of a hypergraph $\mathcal{H}$ is the length of a shortest cycle in $\mathcal{H}$ (if no cycle exists, then we say that the girth of $\mathcal{H}$ is infinity). The \emph{chromatic number} of a hypergraph $\mathcal{H}$ is the minimum number $r$ for which there exists an $r$-coloring of the vertex set of $\mathcal{H}$ with no monochromatic edges. 
	A hypergraph is \emph{$d$-degenerate} if every subhypergraph contains a vertex of degree at most $d$, and we define its \emph{degeneracy} to be the smallest $d$ for which this property holds. 
	
	\begin{proof}[Proof of Theorem~\ref{thm:non-equivalentpairs}\ref{thm:noexternalcycle}]
		Suppose that $H$ contains a copy $K$ of $K_t$, and $H$ contains a cycle with vertices from both $V(K)$ and $V(H)\setminus V(K)$.
		Let $g\geq 3$ denote the length of a shortest such cycle in $H$ and let $k=\lvert E(T)\rvert$.
		
		We shall use a hypergraph of high girth and high minimum degree.
		The existence of such a hypergraph follows from a well-known result of Erd\H{o}s and Hajnal~\cite{erdos_chromatic_1966}; we sketch the argument here for the sake of completeness. Erd\H{o}s and Hajnal~\cite{erdos_chromatic_1966} showed that there exists a $t$-uniform hypergraph $\cH'$ with girth at least $g+1$ and chromatic number at least $kt+1$. It is not difficult to show, using a greedy algorithm, that a $d$-degenerate hypergraph has chromatic number at most $d+1$. Hence, the degeneracy of $\cH'$ must be at least $kt$, and thus $\cH'$ must contain a $t$-uniform subhypergraph $\cH$ with girth at least $g+1$ and minimum degree at least $kt$.

		We construct a graph $F$ with vertex set $V(\cH)$ by embedding a copy of $K_t$ into each edge of $\cH$.
		First observe that $F$ does not contain a copy of $H$, since $\cH$ has girth larger than $g$ and hence each cycle of length at most $g$ is fully contained in one of the copies of $K_t$, that is, in one of the hyperedges, and no two copies of $K_t$ in $F$ share an edge.
		In particular $F\not\to(T,H)$.
		Next we shall prove that $F\to(G,K_t)$.
		Consider a $2$-coloring of $E(F)$ without blue copies of $K_t$.
		Then each copy of $K_t$ in $F$ (each hyperedge of $\cH$) contains a red edge.  Since $\cH$ has minimum degree $kt$ and is $t$-uniform, there are at least $v(\mathcal{H})k$ red edges, that is, the average red degree of $F$ is at least $2k$. It follows from a standard greedy argument that the red subgraph of $F$ contains a subgraph of minimum degree at least $k$. By greedily embedding the vertices of $T$ in this subgraph, we find a red copy of $T$.
		Hence $F\to(T,K_t)$ and $(T,K_t)\not\sim (G,H)$.
	\end{proof}
	
	\subsection{Proof of Theorem~\ref{thm:non-equivalentpairs}\ref{thm:connectedG}}
	In order to prove part~\ref{thm:connectedG} of Theorem~\ref{thm:non-equivalentpairs}, we will use a gadget graph known as a determiner. A graph $D$ with a distinguished edge $\beta\in E(D)$ is called \emph{$(G,H,\beta)$-determiner}, if $D\not\to (G,H)$, and in every $(G,H)$-free red/blue-coloring of $E(D)$, the edge $\beta$ is colored red. Moreover, we call such determiner \emph{well-behaved} if it has an $(G,H)$-free coloring in which all edges incident to $\beta$ are blue. Burr, Erd\H{o}s, Faudree, Rousseau, and Schelp~\cite{BEFRS82} showed that well-behaved determiners exist for any pair $(T,K_t)$ when both the tree and the clique have at least three vertices. In fact, their construction satisfies some further properties which we will use in the proof of Theorem~\ref{thm:non-equivalentpairs}\ref{thm:connectedG}. We summarize those in the following proposition.
	
	\begin{prop}[{\cite[Proof of Theorem 8, Lemmas 9 \& 10]{BEFRS82}}]\label{prop:treeDeterminer}
		Let $t\geq 3$ be an integer and $T$ be a tree with at least three vertices.
		There exists a well-behaved $(T,K_t,\beta)$-determiner $D$. Moreover, the graph induced by the endpoints of $\beta$ and the union of their neighborhoods is isomorphic to $K_t$. 
	\end{prop}
	
	\begin{proof}[Proof of Theorem~\ref{thm:non-equivalentpairs}\ref{thm:connectedG}]
		We may assume that $G\not\subseteq T$, since otherwise $G$ is a tree and we can switch the graphs $G$ and $T$ in the statement. Suppose that the pairs $(T,K_t)$ and $(G,K_t)$ are Ramsey equivalent.
		In order to reach a contradiction, we will construct a graph $F$ which is Ramsey for $(T,K_t)$ but not Ramsey for $(G,K_t)$.
		To do so, first fix a $(T,K_t,\beta)$-determiner $D$, as given by Proposition~\ref{prop:treeDeterminer}. To create $F$, we start with a copy $T_0$ of $T$, and for each edge $e$ of $T_0$ we take a copy $D_e$ of $D$ on a new set of vertices and identify $e$ with the copy of $\beta$ in $D_e$.
		
		We first observe that $F$ is a Ramsey graph for $(T, K_t)$. Indeed, if we assume that $F$ has a $(T,K_t)$-free coloring, then this induces a $(T,K_t)$-free coloring on each copy of $D$, so each copy of $\beta$ needs to be red by the definition of a determiner. But then $T_0$ becomes a red copy of $T$, a contradiction.
		
		It remains to prove that $F$ is not Ramsey for $(G,K_t)$, i.e.,~to find a $(G,K_t)$-free coloring of $E(F)$. For this, fix any edge $e_0\in E(T_0)$.
		We first observe that the graph $F-e_0$ is not Ramsey for $(T,K_t)$ by considering the following coloring: give each copy of $D$ a $(T,K_t)$-free coloring such that its copy of $\beta$ is red (or not colored if $\beta=e_0$) and all edges incident to $\beta$ in $D$ are blue.
		The existence of such a coloring is guaranteed by the fact that $F$ is well-behaved.
		
		By our assumption that $(T,K_t)$ and $(G,K_t)$ are Ramsey equivalent, we conclude that $F - e_0$ is not a Ramsey graph for $(G,K_t)$. Therefore, we can find a $(G, K_t)$-free coloring $c$ of $F-e_0$. We now extend this coloring to $F$ by assigning the color blue to $e_0$.
		If this does not create a blue copy of $K_t$, we have already found the required coloring. 
		So we may assume that this extension leads to a blue copy $K$ of $K_t$. Notice that every copy of $K_t$ in $F$ is fully contained in a copy of the determiner $D$. Then by Proposition~\ref{prop:treeDeterminer} this blue copy of $K_t$ is the graph induced by the endpoints of $e_0$ and the union of their neighborhood in $D_{e_0}$, i.e., it must be contained in the copy $D_{e_0}$ of $D$ and is unique. 
		We now use this information to recolor all other copies of $D-\beta$ in $F$ using the coloring of $E(D_{e_0} - e_0)$; we further color $T_0$ fully red.
		In this new coloring of $E(F)$, there cannot be a blue copy of $K_t$ as there were none in $D_{e_0} - e_0$. 
		Moreover, there cannot be a red copy of $G$, since every copy of $D - \beta$ has a $(G,K_t)$-free coloring, every edge incident to $T_0$ is blue, and $G\not\subseteq T$.
		This is a contradiction to the assumption $F\to (G,K_t)$ and hence $(T,K_t)\not\sim (G,K_t)$.
	\end{proof}

	\section{Concluding remarks and open problems}\label{sec:concluding}
	In this paper we identify a non-trivial infinite family of Ramsey equivalent pairs of connected graphs of the form $(T,K_t) \sim (T,K_t\cdot K_2)$, where $T$ is a non-trivial star or a so-called suitable caterpillar.
	We also prove that $(T,K_t) \not\sim (T,K_t\cdot K_2)$  for a large class of other trees $T$ including all trees of odd diameter.
	It remains open whether for the remaining trees the respective pairs are Ramsey equivalent or not.	
	Our proof actually shows $(G,K_t)\sim (G,K_t\cdot K_2)$ for all so-called woven graphs $G$ and sufficiently large $t$.
	This leads to the following two questions:
	Are there any woven graphs other than the trees mentioned in Theorem~\ref{thm:equivalentpairs}\ref{thm:StarCaterpillar}? Are there non-woven graphs $G$ and integers $t$ with $(G,K_t)\sim (G,K_t\cdot K_2)$?

	\smallskip

	One of the questions that drove the study of Ramsey equivalence is: What graphs $H$ are Ramsey equivalent to the clique $K_t$? This question was addressed in~\cite{bloom2018ramsey,fox2014ramsey,szabo2010minimum}. In particular, it follows from the results of Folkman~\cite{folkman1970graphs} and Nešetřil and Rödl~\cite{NESETRIL1976243} and Fox, Grinshpun, Liebenau, Person, and Szabó~\cite{fox2014ramsey} that there is no connected graph $H \neq K_t$ such that $H\sim K_t$. It is then natural to ask: what about an asymmetric pair of connected graphs?
	
	\begin{question}\label{quest:clique}
		Are there connected graphs $G$ and $H$ and an integer $t$ such that $(G,H)\sim (K_t, K_t)$?
	\end{question}

	Some known results allow us to easily exclude many possible pairs $(G,H)$. For example, the results of Folkman~\cite{folkman1970graphs} and Nešetřil and Rödl~\cite{NESETRIL1976243}, as stated in Theorem~\ref{thm:NRcliquenumber} above, show that, if $\max\{\omega(G), \omega(H)\}\neq t$, then $(G,H)\not\sim (K_t,K_t)$, while the work of Fox, Grinshpun, Liebenau, Person, and Szabó~\cite{fox2014ramsey} shows that we cannot have $\omega(G) = \omega(H) = t$. Thus, we can restrict our attention to pairs $(G,H)$ with $\omega(G) <t$ and $\omega(H) = t$. Combining several results concerning Ramsey properties of the random graph $G(n,p)$~\cite{B17,HST19,mousset2018towards,rodl1993lower,rodl1995threshold}, we can restrict $(G,H)$ even further: namely, we can show that $m_2(G) = m_2(H) = m_2(K_t)$. Using the ideas developed by Savery in~\cite{savery2022chromatic}, we can also prove that the chromatic numbers of the graphs $G$ and $H$ must satisfy either $\chi(G) = t-1$ and $\chi(H) = t+1$, or $\chi(G) = t$ and $H = K_t$. In addition,
	the theory of determiners developed in~\cite{burr1985useofsenders} for 3-connected graphs allows us to conclude that $G$ and $H$ cannot both be 3-connected. 
	It would be very interesting to provide a complete answer to Question~\ref{quest:clique}. 
	
	\smallskip
	Our study focuses on pairs of connected graphs.
	Disconnected graphs have also received some attention in the symmetric setting; the central question here asks which graphs are Ramsey equivalent to a complete graph~\cite{bloom2018ramsey,fox2014ramsey,szabo2010minimum}.
	Similar questions arise in the asymmetric setting, for instance for which graphs $G$ and integers $t$ we have $(G,K_t)\sim (G,K_t+K_{t-1})$, where $K_t+K_{t-1}$ is the disjoint union of $K_t$ and $K_{t-1}$ (this holds in case $G=K_t$ by~\cite{bloom2018ramsey}).

	\bibliographystyle{amsplain}
	\bibliography{biblio}

\end{document}